\documentclass[11pt, a4paper]{article}

\usepackage{latexsym}
\usepackage{amsmath, epsfig}
\usepackage{color}
\usepackage{graphicx, float}
\usepackage{amssymb,amsfonts,amsthm,amsmath,graphicx,url}
\usepackage[T1]{fontenc}
\usepackage[utf8]{inputenc}
\usepackage{authblk}

\definecolor{VeryLightBlue}{rgb}{0.9,0.9,1}
\definecolor{LightBlue}{rgb}{0.8,0.8,1}
\definecolor{MidBlue}{rgb}{0.5,0.5,1}
\definecolor{DarkBlue}{rgb}{0,0,0.6}
\definecolor{Blue}{rgb}{0,0,1}
\definecolor{Gold}{rgb}{1,0.843,0}
\definecolor{LightGreen}{rgb}{0.88,1,0.88}
\definecolor{MidGreen}{rgb}{0.6,1,0.6}
\definecolor{DarkGreen}{rgb}{0,0.6,0}
\definecolor{VeryLightYellow}{rgb}{1,1,0.9}
\definecolor{LightYellow}{rgb}{1,1,0.6}
\definecolor{MidYellow}{rgb}{1,1,0.5}
\definecolor{DarkYellow}{rgb}{1,1,0.2}
\definecolor{DarkPurple}{rgb}{.6,0,1}
\definecolor{Red}{rgb}{1,0,0}
\definecolor{VeryLightRed}{rgb}{1,0.9,0.9}
\definecolor{LightRed}{rgb}{1,0.8,0.8}
\definecolor{MidRed}{rgb}{1,0.55,0.55}

\long\def\delete#1{}

\usepackage[normalem]{ulem}

\newtheorem{theorem}{Theorem}[section]
\newtheorem{lemma}[theorem]{Lemma}

\input amssym.def
\input amssym.tex

\newcommand{\be}{\begin{equation}}
\newcommand{\ee}{\end{equation}}
\newcommand{\bea}{\begin{eqnarray}}
\newcommand{\eea}{\end{eqnarray}}
\newcommand{\bean}{\begin{eqnarray*}}
\newcommand{\eean}{\end{eqnarray*}}

\def\qed{\hfill$\Box$\vspace{11pt}}

\def\ZZZ{\mathbb{Z}}

\def\b0{{\bf 0}}

\def\Ga{\Gamma}

\def\b{\beta}

\def\l{\lambda}

\def\z{\zeta}

\def\Cay{{\rm Cay}}

\def\SL{{\rm SL}}

\def\Cay{{\rm Cay}}

\title{Perfect codes in circulant graphs}

\author[a]{Rongquan Feng}
\author[b]{He Huang}
\author[c]{Sanming Zhou}

\affil[a]{LMAM, School of Mathematical Sciences, Peking University, Beijing 100871, China\\ 

Email: fengrq@math.pku.edu.cn}

\affil[b]{School of Mathematical Sciences, Peking University, Beijing 100871, China\\

Email: 1301110019@math.pku.edu.cn
}

\affil[c]{School of Mathematics and Statistics, The University of Melbourne, Parkville, VIC 3010, Australia\\

Email: sanming@unimelb.edu.au}

\date{}

\begin{document}
\openup 0.5\jot
\maketitle


\begin{abstract}
A perfect code in a graph $\Ga = (V, E)$ is a subset $C$ of $V$ that is an independent set such that every vertex in $V \setminus C$ is adjacent to exactly one vertex in $C$. A total perfect code in $\Ga$ is a subset $C$ of $V$ such that every vertex of $V$ is adjacent to exactly one vertex in $C$. A perfect code in the Hamming graph $H(n, q)$ agrees with a $q$-ary perfect 1-code of length $n$ in the classical setting. In this paper we give a necessary and sufficient condition for a circulant graph of degree $p-1$ to admit a perfect code, where $p$ is an odd prime. We also obtain a necessary and sufficient condition for a circulant graph of order $n$ and degree $p^l-1$ to have a perfect code, where $p$ is a prime and $p^l$ the largest power of $p$ dividing $n$. Similar results for total perfect codes are also obtained in the paper. 


\end{abstract}

\section{Introduction}
\label{sec:int}

Since the beginning of coding theory in the late 1940s, perfect codes have been important objects of study in information theory; see the surveys \cite{Heden, vanLint} for a large number of results on perfect codes. Hamming and Golay codes are well known examples of perfect codes, and their importance is widely recognized. The notion of perfect codes can be generalized to graphs \cite{Biggs73, Kra} in a natural way, such that $q$-ary perfect $e$-codes of length $n$ in the classical setting are precisely perfect $e$-codes in the corresponding Hamming graph $H(n, q)$. Since Hamming graphs are a particular family of Cayley graphs, perfect codes in Cayley graphs can be viewed as generalizations of perfect codes in the classical case. From a group theoretic point of view, the simplest Cayley graphs are circulant graphs, namely Cayley graphs on cyclic groups. However, even in this innocent-looking case, the question about when a general circulant graph admits a perfect 1-code is unsettled. Contributing to improvement of this unsatisfactory situation, we answer this question for two families of circulant graphs and give similar results for total perfect codes in this paper. 

Let $\Ga = (V, E)$ be a simple undirected graph and $e \ge 1$ an integer. The \emph{ball} with radius $e$ and centre $u \in V$ is the set of vertices of $\Ga$ with distance at most $e$ to $u$ in $\Ga$. A subset $C$ of $V$ is called a \emph{perfect $e$-code} \cite{Biggs73, Kra} in $\Ga$ if the balls with radius $e$ and centres in $C$ form a partition of $V$. As mentioned above, $q$-ary perfect $e$-codes of length $n$ in the classical setting \cite{Heden, vanLint} are simply perfect $e$-codes in the Hamming graph $H(n, q)$. In graph theory, perfect $1$-codes in a graph are called efficient dominating sets or independent perfect dominating sets of the graph. In the rest of this paper a perfect 1-code is simply called a \emph{perfect code}. A subset $C \subseteq V$ is called a \emph{total perfect code} in $\Ga$ (see e.g. \cite{GHT}) if every vertex of $\Ga$ has exactly one neighbour in $C$. This concept is related to diameter perfect codes, which were introduced in \cite{AAK01} for distance regular graphs and adapted in \cite{Etzion11} for Lee metric over $\mathbb{Z}^n$ and $\mathbb{Z}_q^n$. As mentioned in \cite{Z16}, when the Manhattan (for $\mathbb{Z}^n$) or Lee (for $\mathbb{Z}_q^n$) distance is considered, total perfect codes coincide with diameter perfect codes of minimum distance four. 
 
Perfect codes in Cayley graphs are particularly charming objects. Given a finite group $G$ and an inverse-closed subset $X$ of $G$ not containing the identity element, the \emph{Cayley graph} $\Cay(G, X)$ on $G$ relative to the \emph{connection set} $X$ is the graph  with vertex set $G$ such that $u, v \in G$ are adjacent if and only if $vu^{-1} \in X$. This graph is connected if and only if $S$ is a generating set of $G$. In the special case when $G=\ZZZ_n$ is the additive group of integers modulo $n$, a Cayley graph $\Cay(\ZZZ_n, S)$ on $\ZZZ_n$ is called a \emph{circulant graph}. In \cite{MBG07} sufficient conditions for Gaussian and Eisenstein-Jacobi graphs to contain perfect $e$-codes were given, and these conditions were proved to be necessary in \cite{Z15} in a more general setting. In \cite{T04} it was proved that there is no perfect code in any Cayley graph on $\SL(2, 2^f)$, $f > 1$ with respect to a conjugation-closed connection set. In \cite{DS03} a methodology for constructing infinite families of E-chains of Cayley graphs on symmetric groups was given, where an E-chain is a countable family of nested graphs each containing a perfect code. In \cite{E87} perfect codes in a Cayley graph with a conjugation-closed connection set were studied by way of equitable partitions, yielding a nonexistence result in terms of irreducible characters of the underlying group. In \cite{L01} it was proved that a conjugation-closed subset $C$ of a group $G$ is a perfect code in a Cayley graph on $G$ if and only if there exists a covering projection from the Cayley graph to a complete graph with $C$ as a fibre. A similar result was obtained in \cite{Z16} for total perfect codes in Cayley graphs. In a recent work \cite{HXZ}, perfect codes in Cayley graphs were studied from the viewpoint of group rings, and among other results conditions for a normal subgroup of a finite group to be a perfect code in some Cayley graph of the group were obtained.
 
Perfect codes in circulant graphs have been studied by several researchers in recent years. In \cite{OPR07}, 3- and 4-regular connected circulant graphs admitting a perfect code were characterized, and a sufficient condition for a general circulant graph to have a perfect code was given. In \cite{YPD14} a necessary and sufficient condition for a circulant graph to admit a perfect code with size a prime number was given and all such perfect codes were characterized. In \cite{KM13} a few results on perfect codes in circulant graphs were proved. In \cite{Z15} perfect $e$-codes in an interesting family of circulant graphs with degree twice an odd prime were studied in the more general setting of cyclotomic graphs. 

In spite of the efforts above, our understanding of perfect codes in circulant graphs is still quite limited. In this paper we prove the following results with the help of cyclotomic polynomials. 
 
\begin{theorem}  
\label{main}    
Let $n$ be a positive integer and $p$ be an odd prime. A connected circulant graph $\Cay(\ZZZ_n, S)$ of degree $p-1$ admits a perfect code if and only if $p$ divides $n$ and $s \not \equiv s'\mod p$ for distinct $s, s' \in S \cup \{0\}$.           
\end{theorem}
 
\begin{theorem}   
\label{th:pl}    
Let $n, l$ be positive integers, and let $p$ be a prime such that $p^l$ divides $n$ but $p^{l+1}$ does not divide $n$. A connected circulant graph $\Cay(\ZZZ_n, S)$ of degree $p^l-1$ admits a perfect code if and only if $s \not \equiv s'\mod p^l$ for distinct $s, s' \in S \cup \{0\}$.          
\end {theorem}

\begin{theorem}
\label{th:main-total}       
Let $n$ be a positive integer and $p$ be an odd prime. A connected circulant graph $\Cay(\ZZZ_n, S)$ of degree $p$ admits a total perfect code if and only if $p$ divides $n$ and $s \not \equiv s'\mod p$ for distinct $s, s' \in S$. 
\end {theorem}

\begin{theorem}
\label{th:pl-total}       
Let $n, l$ be positive integers, and let $p$ be a prime such that $p^l$ divides $n$ but $p^{l+1}$ does not divide $n$. A connected circulant graph $\Cay(\ZZZ_n, S)$ of degree $p^l$ admits a total perfect code if and only if $s \not \equiv s'\mod p^l$ for distinct $s, s' \in S$.
\end {theorem}

\section{Proofs}
\label{sec:pf}

Let $\zeta_n$ be a primitive $n$th root of unity, say $\z_n = e^{2\pi i/n}$. The \emph{$n$th cyclotomic polynomial} is defined \cite{J14} as 
$$ 
\l_n(x) = \prod_{1 \le d < n, (d, n) = 1} (x-\z_n^d).
$$
The roots of $\l_n(x)$ are precisely the primitive $n$th roots of unity, that is, $\lambda_{n} (x) = \prod_{\z\in E_n} (x-\z)$, where $E_n$ is the set of all primitive $n$th roots of unity. We will use the following well known results (see e.g. \cite[Section 9.1]{J14}) in the proof of Theorems 1.1-1.4. 

\begin{lemma}     
\label{lem:cyc}
\begin{itemize} 
\item[\rm (a)] 
\be
\label{eq:cyc}
x^n-1=\prod_{d \mid n}\lambda_{d} \left( x \right );
\ee
\item[\rm(b)]
$\lambda_{n}(x) \in \ZZZ[x]$;
\item[\rm(c)]
$\lambda_{n}(x)$ is irreducible in $\ZZZ[x]$.
\end{itemize}
\end{lemma}

In particular, by \eqref{eq:cyc}, for any prime $p$ and integer $j \ge 1$,
\be
\label{eq:pj}
\lambda_{p^j}(x) = \frac{x^{p^j}-1}{x^{p^{j-1}}-1}=(x^{p^{j-1}})^{p-1}+ (x^{p^{j-1}})^{p-2}+\cdots  + x^{p^{j-1}}+1.
\ee

Define
$$
f_A(x) = \sum_{a \in A} x^a
$$
for any non-empty finite set $A$ of nonnegative integers. For a subset $S$ of $\ZZZ_n$, denote
$$
S_0 = S \cup \{0\}.
$$ 
The following lemma reduces the perfect code problem for circulant graphs to a number theoretic problem.  

\begin{lemma}
\label{lem:equiv-def}
A subset $C$ of $\ZZZ_n$ is a perfect code in $\Cay(\ZZZ_n, S)$ if and only if there exists $q(x)\in \ZZZ[x]$ such that
\be   
\label{eq:ff}      
f_C(x) f_{S_0}(x) = (x^n-1) q(x) + (x^{n-1}+\cdots +x+1).
\ee
\end{lemma}

\begin{proof}
By the definition of a perfect code, $C$ is a perfect code in $\Cay(\ZZZ_n, S)$ if and only if every integer in $\{0, 1, \ldots, n-1\}$ can be written in a unique way as $(c+s) \mod n$ with $c \in C$ and $s\in S_0$, which is equivalent to $f_C(x) f_{S_0}(x) = \sum_{c \in C, s \in S_0} x^{c+s} \equiv x^{n-1} + \cdots + x + 1 \mod (x^n-1)$. Thus $C$ is a perfect code in $\Cay(\ZZZ_n, S)$ if and only if \eqref{eq:ff} holds for some $q(x) \in \ZZZ[x]$. 
\qed
\end{proof}

The next lemma was proved in \cite[Remark 1]{OPR07}. We give a different proof using Lemma \ref{lem:equiv-def} for the completeness of the present paper. 

\begin{lemma}
\label{lem:suf}
(\cite[Remark 1]{OPR07})  
A connected circulant graph $\Cay(\ZZZ_n, S)$ of order $n\ge 4$ and degree $k = |S|$ admits a perfect code provided that $k+1$ divides $n$ and $s \not \equiv s' \mod\  (k+1)$ for distinct $s, s' \in S \cup \{0\}$.
\end{lemma}
 
\begin{proof}   
Consider a connected circulant graph $\Cay(\ZZZ_n, S)$ with order $n\ge 4$ and degree $k = |S|$. Suppose that $k+1$ divides $n$, say $n=m(k+1)$ for some integer $m \ge 1$, and $s \not \equiv s' \mod\  (k+1)$ for distinct $s, s' \in S \cup \{0\}$. We may write $S_0=S \cup \{0\} = \{s_0, s_1, \ldots, s_k\}$, where $s_0, s_1, \ldots, s_k$ are pairwise distinct modulo $(k+1)$. Without loss of generality we may assume $s_0=0$ and $s_i  \equiv i \mod\  (k+1)$ for $i \in \{1, 2, \ldots, k\}$. Then $x^{s_i}\equiv \ x^i \mod\  (x^{k+1}-1)$ for $i \in \{0, 1, \ldots, k\}$ and $f_{S_0}(x)=\sum_{s \in S_0} x^s \equiv x^k + \cdots + x + 1 \mod (x^{k+1}-1)$. So there exists $q(x)\in \ZZZ[x]$ such that $f_{S_0}(x)=(x^{k+1}-1) q(x)+( x^k + \cdots + x + 1)$. 

Set $C=\{0,k+1,2(k+1),\ldots,(m-1)(k+1)\}$. Then 
$$
f_C(x)=x^{(m-1)(k+1)}+ \cdots + x^{k+1} + 1=\frac{x^{m(k+1)}-1}{x^{k+1}-1}=\frac{x^n-1}{x^{k+1}-1}
$$
and hence
$$
f_C(x) f_{S_0}(x) = (x^n-1) q(x) + (x^{n-1}+\cdots +x+1).
$$
Therefore, by Lemma \ref{lem:equiv-def}, $C$ is a perfect code in $\Cay(\ZZZ_n, S)$.  
\qed
\end{proof}

As shown in \cite{KM13, OPR07} by counterexamples, the sufficient condition for the existence of a perfect code in $\Cay(\ZZZ_n, S)$ given in Lemma \ref{lem:suf} may not be necessary. However, it is indeed necessary when $k=4$ (see \cite{OPR07}) or when $n/(k+1)$ is a prime and $S \cup \{0\}$ is aperiodic (see \cite{YPD14} for definition).

\medskip  
\begin{proof}{\bf of Theorem \ref{main}}~
By Lemma \ref{lem:suf}, it remains to prove the `only if' part. Suppose that $\Cay(\ZZZ_n, S)$ is connected of degree $|S| = p-1$ and admits a perfect code $C$, where $p$ is an odd prime. Then by Lemma \ref{lem:equiv-def}, \eqref{eq:ff} holds for some $q(x)\in \ZZZ[x]$. Setting $x=1$ in \eqref{eq:ff}, we obtain $p |C| = |C| |S_0| = f_C(1) f_{S_0}(1) = n$. Hence $p$ divides $n$. Write $n=p^l m$ with $l \ge 1$ and $m$ not divisible by $p$. Then $|C| = p^{l-1} m$ and $p^l$ does not divide $|C|$. 
 
By Lemma \ref{lem:cyc}, $\lambda_{p}(x), \lambda_{p^2}(x), \ldots , \lambda_{p^l}(x)$ are distinct irreducible polynomials each dividing $x^n-1$ and $(x^n-1)/(x-1) = x^{n-1}+\cdots +x+1$. Combining this with (\ref{eq:ff}), we obtain that $\lambda_{p^j}(x)$ divides $f_C(x)$ or $f_{S_0}(x)$ for each $j \in \{1, 2, \ldots, l\}$. 

\smallskip
\textsf{Claim 1:} There exists at least one $j \in \{1, 2, \ldots, l\}$ such that $\lambda_{p^j}(x)$ divides $f_{S_0}(x)$.

Suppose otherwise. Then $\lambda_{p}(x), \lambda_{p^2}(x), \ldots , \lambda_{p^l}(x)$ all divide $f_C(x)$. Since they are irreducible and hence pairwise coprime, it follows that $\prod_{j=1}^l \lambda_{p^j}(x)$ divides $f_C(x)$. That is,
\be  
\label{f(D)}
f_C(x) = g(x) \prod_{j=1}^l \lambda_{p^j}(x) 
\ee
for some $g(x)\in \ZZZ[x]$. Since $p$ is a prime, by \eqref{eq:pj} we have $\l_{p^j}(1) = p$ for each $j \ge 1$. Setting $x=1$ in (\ref{f(D)}), we then obtain $|C| = f_C(1) = p^l \cdot g(1)$. Since $g(1)$ is an integer, it follows that $p^{l}$ divides $|C|$, which is a contradiction. This proves Claim 1. 

Since $|S_0| = p$, we may write $S_0 = \{s_0, s_1, \ldots, s_{p-1}\}$, where $s_0=0$ and $s_1, s_2, \ldots, s_{p-1}$ are pairwise distinct. Denote by $t_i$ the unique integer in $\{0,1,\ldots, p-1\}$ such that $s_i \equiv t_i \mod p$, for $0 \le i \le p-1$. In particular, $t_0 = 0$ as $s_0 = 0$. We have
\be  
\label{fs}    
f_{S_0}(x) =\sum^{p-1}_{i=0} x^{s_i} \equiv \sum^{p-1}_{i=0} x^{t_i} \mod (x^p-1).  
\ee

\textsf{Claim 2:}  If $j \in \{2, \ldots, l\}$, then $\lambda_{p^j}(x)$ does not divide $f_{S_0}(x)$. 

Suppose to the contrary that $\lambda_{p^j}(x)$ divides $f_{S_0}(x)$ for some $j \in \{2, \ldots, l\}$, say, $f_{S_0}(x) = \lambda_{p^j}(x) h(x)$, where $h(x)\in \ZZZ[x]$. Since $j \ge 2$, $x^{p^{j-1}} \equiv 1 \mod (x^p-1)$. This together with \eqref{eq:pj} implies $\lambda_{p^j} ( x ) \equiv p \mod (x^p-1)$.
Thus, 
\be
\label{ph}
f_{S_0}(x) \equiv p\cdot \overline{h}(x) \mod (x^p-1),
\ee
where $\overline{h}(x)$ is the unique polynomial of degree less than $p$ such that $h(x) \equiv \overline{h}(x) \mod (x^p-1)$. Combining \eqref{fs} and \eqref{ph}, we have 
$$
\sum^{p-1}_{i=0} x^{t_i} \equiv p\cdot \overline{h}(x) \mod (x^p-1).  
$$
Since both sides of this equation are polynomials of degree less than $p$ whilst $x^p-1$ has degree $p$, it follows that 
\be
\label{eq:add}
\sum^{p-1}_{i=0} x^{t_i} = p\cdot \overline{h}(x).
\ee
Since $0 \le t_i \le p-1$ and $t_0=0$, this implies that all $t_i =0$, that is, every element of $S$ is a multiple of $p$. However, this implies that $\Cay(\ZZZ_n, S)$ is disconnected, which contradicts our assumption. This proves Claim 2. 
 
Combining Claims 1 and 2, we know that $\lambda_{p}(x)$ divides $f_{S_0}(x)$. Thus, by \eqref{fs} and $x^p-1=(x-1) \lambda_p (x)$,  we obtain $\sum^{p-1}_{i=0} x^{t_i} \equiv 0 \mod \lambda_{p}(x)$. Since $\sum^{p-1}_{i=0} x^{t_i}$ has degree at most $p-1$ whilst $\lambda_{p}(x)$ has degree $p-1$, it follows that $\sum^{p-1}_{i=0} x^{t_i} = a \lambda_{p}(x)$ for some integer $a$. Setting $x=1$, we obtain $p=ap$ and so $a=1$. Therefore, $\sum^{p-1}_{i=0} x^{t_i} = \lambda_{p}(x) =x^{p-1}+\cdots+x+1$. In other words, $\{t_0, t_1, \ldots, t_{p-1}\} = \{0,1,\ldots,p-1\}$, or equivalently, $s \not= s' \mod p$ for distinct $s, s' \in S_0$.  
\qed
\end{proof}
 
\begin{proof}{\bf of Theorem \ref{th:pl}}~
Again, by Lemma \ref{lem:suf}, it remains to prove the `only if' part. Suppose that $\Cay(\ZZZ_n, S)$ is connected of degree $|S| = p^l-1$ and admits a perfect code $C$, where $n , l$ are positive integers and $p$ a prime such that $p^l$ divides $n$ but $p^{l+1}$ does not. Then by Lemma \ref{lem:equiv-def}, \eqref{eq:ff} holds for some $q(x)\in \ZZZ[x]$. 
Setting $x=1$ in \eqref{eq:ff}, we obtain $p^l |C| = |C| |S_0| = f_C(1) f_{S_0}(1) = n$. Since $p^{l+1}$ does not divide $n$, $p$ does not divide $|C|$.  

Similar to the proof of Theorem \ref{main}, we see that $\lambda_{p^j}(x)$ divides $f_C(x)$ or $f_{S_0}(x)$ for each $j \in \{1, 2, \ldots, l\}$. We prove further that:

\smallskip
\textsf{Claim 3:} $\lambda_{p^j}(x)$ divides $f_{S_0}(x)$ for each $j \in \{1, 2, \ldots, l\}$. 

To prove this, let $J$ denote the set of integers $j \in \{1, 2, \ldots, l\}$ such that $\lambda_{p^j}(x)$ divides $f_C(x)$. Since $\lambda_{p}(x), \lambda_{p^2}(x), \ldots , \lambda_{p^l}(x)$ are irreducible and hence pairwise coprime, 
\be  
\label{fC}
f_C(x) = g(x) \prod_{j \in J} \lambda_{p^j}(x) 
\ee
for some $g(x) \in \ZZZ[x]$. By \eqref{eq:pj}, we have $\l_{p^j}(1) = p$ for each $j \ge 1$. Thus, setting $x=1$ in (\ref{fC}), we obtain $|C| = f_C(1) = p^{|J|} \cdot g(1)$. Since $g(1)$ is an integer, it follows that $p^{|J|}$ divides $|C|$. Since $p$ does not divide $|C|$, we must have $J = \emptyset$ and so Claim 3 is proved. 

Since $|S_0| = p^l$, we may write $S_0 = \{s_0, s_1, \ldots, s_{p^l-1}\}$, where $s_0=0$. Denote by $t_i$ the unique integer in $\{0,1,\ldots, p^l-1\}$ such that $s_i \equiv t_i \mod p^l$, for $0 \le i \le p^l-1$. In particular, $t_0 = 0$ as $s_0 = 0$. We have
\be  
\label{fs1}    
f_{S_0}(x) =\sum^{p^l-1}_{i=0} x^{s_i} \equiv \sum^{p^l-1}_{i=0} x^{t_i} \mod (x^{p^l}-1).  
\ee 
On the other hand, by Claim 3,  
\be
\label{eq:fS1}
f_{S_0}(x) = h(x) \prod_{j=1}^l \lambda_{p^j}(x)
\ee
for some $h(x)\in \ZZZ[x]$. By \eqref{eq:cyc}, $\prod_{j=1}^l \l_{p^j}(x) = (x^{p^l}-1)/(x-1) = \sum_{j=0}^{p^l-1} x^j$ divides $x^{p^l}-1$. This together with \eqref{fs1} and \eqref{eq:fS1} implies that $\sum_{j=0}^{p^l-1} x^j$ divides $\sum^{p^l-1}_{i=0} x^{t_i}$. Since the former has degree $p^{l}-1$ whilst the latter has degree at most $p^l - 1$, it follows that the latter must have degree $p^l - 1$ and moreover $\sum^{p^l-1}_{i=0} x^{t_i} =a \sum_{j=0}^{p^l-1} x^j$ for some integer $a$. Setting $x=1$, we obtain $p^l=ap^l$ and so $a=1$. That is, $\sum^{p^l-1}_{i=0} x^{t_i} = \sum_{j=0}^{p^l-1} x^j$. Therefore, $\{t_0, t_1, \ldots, t_{p^l-1}\} = \{0,1,\ldots, p^l-1\}$ and the proof is complete. 
\qed    
\end{proof}

Similar to Lemma \ref{lem:equiv-def}, one can easily verify the following result. 
\begin{lemma}
\label{lem:total-def}
A subset $C$ of $\ZZZ_n$ is a total perfect code in $\Cay(\ZZZ_n, S)$ if and only if there exists $q(x) \in \ZZZ[x]$ such that
\begin{equation}   
\label{ff1}
f_C(x) f_{S}(x) = (x^n-1)q(x) + (x^{n-1} + \cdots + x + 1).
\end{equation}
\end{lemma}

Similar to \cite[Remark 1]{OPR07} (see Lemma \ref{lem:suf}), we have the following observation.
\begin{lemma}
\label{lem:suf-total}
A connected circulant graph $\Cay(\ZZZ_n, S)$ of order $n\ge 4$ and degree $k = |S|$ admits a total perfect code provided that $k$ divides $n$ and $s \not \equiv s' \mod k$ for distinct $s, s' \in S$.
\end{lemma}

In fact, since all elements of $S$ are pairwise distinct modulo $k$, for any $v \in \ZZZ_n$ there exists a unique $s\in S$ such that $v \equiv s \mod k$, implying that $\{k i: 0 \le i < n/k\}$ is a total perfect code in $\Cay(\ZZZ_n, S)$. 

Theorem \ref{th:main-total} can be proved using Lemmas \ref{lem:total-def} and \ref{lem:suf-total} and following the proof of Theorem \ref{main} but with $S \cup \{0\}$ replaced by $S = \{s_0, s_1, \ldots, s_{p-1}\}$. (Since $s_0 \ne 0$ in the current case, $t_0$ may not be $0$, but one can see that not all elements of $S$ are congruent to each other modulo $p$ as $p$ is odd and each $n - s_i \in S$. So from \eqref{eq:add} we can still derive that all $t_i = 0$.)

Theorem \ref{th:pl-total} can be proved using Lemmas \ref{lem:total-def} and \ref{lem:suf-total} and following the proof of Theorem \ref{th:pl} but with $S \cup \{0\}$ replaced by $S = \{s_0, s_1, \ldots, s_{p^l - 1}\}$.

\section{Remarks}
\label{sec:rem}

We remark that our method in the previous section can be adapted to give a totally different proof of the following known result. 

\begin{theorem} 
\label{th:4}  
(\cite[Theorem 2]{OPR07})   
A cubic connected circulant graph of order $n \ge 4$ admits a perfect code if and only if $n \equiv 4 \mod 8$.
\end{theorem}   
  
\begin{proof} 
It can be verified that, for a cubic connected circulant graph $\Cay(\ZZZ_n, S)$, where $n \ge 4$ and $S = \{n/2, s, n-s\}$ for some $1 \le s \le (n/2)-1$, $n$ must be even and moreover $n \equiv 4 \mod 8$ if and only if $4$ divides $n$ and the elements of $S \cup \{0\}$ are pairwise distinct modulo $4$. Thus, by Lemma \ref{lem:suf}, if $n \equiv 4 \mod 8$, then $\Cay(\ZZZ_n, S)$ admits a perfect code $C$.

Suppose that $\Cay(\ZZZ_n, S)$ admits a perfect code $C$. Similar to the proof of Theorem \ref{main}, by Lemma \ref{lem:equiv-def}, (\ref{eq:ff}) holds and so $4|C| = |C| |S_0| = f_C(1) f_{S_0}(1) = n$. This together with the connectedness of $\Cay(\ZZZ_n, S)$ implies that $n$ is a multiple of 4 and $s$ must be odd.  Write $n=2^l m$ with $l \ge 2$ and $m$ odd. Then $|C| = 2^{l-2} m$. Since $2^{l-1}$ does not divide $|C|$, similar to Claim 1 in the proof of Theorem \ref{main} one can show that exactly two of $\lambda_{2}(x), \lambda_{2^2}(x), \ldots, \lambda_{2^l}(x)$ divide $f_{S_0}(x)$. So there is at least one $j \in \{2, \ldots, l\}$ such that $\lambda_{2^j}(x)$ divides $f_{S_0}(x)$. Note that $\lambda_{2^j} (x) = x^{2^{j-1}}+1$ by \eqref{eq:pj}. 

Write $s = 2^{j-1} q + r$, where $q$ and $r$ are integers and $0 \le r \le 2^{j-1}-1$. Since $s$ is odd and $j \ge 2$, $r$ is odd and so $1 \le r \le 2^{j-1}-1$.  We have 
\bea   
f_{S_0}(x) & = & x^0+x^{n/2} +x^s+x^{n-s} \nonumber  \\
& = & 1+(x^{2^{j-1}})^{2^{l-j}m}+(x^{2^{j-1}})^{q} \cdot x^{r}+(x^{2^{j-1}})^{2^{l-j+1}m-q-1}  \cdot x^{2^{j-1}-r} \nonumber \\   
& \equiv & 1+(-1)^{2^{l-j}m}+(-1)^{q} \cdot x^{r}+(-1)^{2^{l-j+1}m-q-1}  \cdot x^{2^{j-1}-r}  \mod (x^{2^{j-1}}+1)  \nonumber \\
& \equiv & 1+(-1)^{2^{l-j}m}+(-1)^{q}  \cdot (x^{r} - x^{2^{j-1}-r})  \mod (x^{2^{j-1}}+1). \nonumber
\eea
Thus, since $x^{2^{j-1}}+1$ divides $f_{S_0}(x)$, it also divides $1+(-1)^{2^{l-j}m}+(-1)^{q}  \cdot (x^{r} - x^{2^{j-1}-r})$. However, this polynomial has degree at most $2^{j-1}-1$ as $1 \le r \le 2^{j-1}-1$. Therefore, $1+(-1)^{2^{l-j}m}+(-1)^{q}  \cdot (x^{r} - x^{2^{j-1}-r})=0$, yielding $1+(-1)^{2^{l-j} m}=0$ and $r = 2^{j-1}-r$. Hence $l=j$ and $r = 2^{j-2}$. Since $r$ is odd, we then have $l=j=2$. So $n=4m$ with odd $m$. Thus $n \equiv 4 \mod 8$ and the proof is complete.         
\qed    
\end{proof}

It is well known that Cayley graphs are vertex-transitive. In general, perfect codes in vertex-transitive graphs are also of considerable interest. For example, the problem of characterizing vertex-transitive graphs admitting a perfect code was posed in \cite{KP12}. In the same paper it was proved that a connected cubic vertex-transitive graph of order $2^m$ ($m\ge 3$) has a perfect code if and only if it is not isomorphic to the M\"{o}bius ladder $M_{2^{m-1}}$. (Since $M_{2^{m-1}}$ is isomorphic to the cubic circulant graph $\Cay(\ZZZ_{2^m}, \{2^{m-1},1,-1\})$ and $2^m$ is divisible by $8$ when $m \ge 3$, the fact that $M_{2^{m-1}}$ has no perfect codes can be thought as a special case of Theorem \ref{th:4}.) Since $2^m \equiv 0 \mod 8$ when $m \ge 3$, this implies that, in contrast to Theorem \ref{th:4}, a connected cubic Cayley graph of order $n$ admitting a perfect code may not satisfy $n \equiv 4 \mod 8$, as shown also in \cite[Table 1]{KP12}. It would be interesting to give a characterization of cubic Cayley graphs (or cubic vertex-transitive graphs) admitting at least one perfect code.

\medskip

\textbf{Acknowledgements}~~R. Feng was supported by NSFC with No. 61370187 and by NSFC--Genertec Joint Fund For Basic Research with No. U1636104 and H. Huang by the China Scholarship Council (No. 201506010015). S. Zhou acknowledges the support of the Australian Research Council (FT110100629). 
The authors are grateful to the two anonymous referees for their helpful comments and suggestions, and to Professor M. Buratti for informing us the recent work \cite{DSLW16} on the same topic using different approaches.


{\small

\begin{thebibliography}{99} 

\bibitem{AAK01}
R. Ahlswede, H. K. Aydinian and L. H. Khachatrian, On perfect codes and related concepts, {\em Des.  Codes Cryptogr.} {\bf 22} (2001), 221--237.
 
\bibitem{Biggs73}
N. Biggs, Perfect codes in graphs, {\em J. Combin. Theory Ser. B} {\bf 15} (1973), 288--296.

\bibitem{DS03}
I. J. Dejter and O. Serra, Efficient dominating sets in Cayley graphs, 
\textit{Discrete Appl. Math.} {\bf 129} (2003), 319--328.

\bibitem{YPD14}
Y-P. Deng, Efficient dominating sets in circulant graphs with domination number prime, {\em Inform. Process. Lett.} {\bf 114} (2014), 700--702. 

\bibitem{DSLW16}
Y-P. Deng, Y-Q. Sun, Q. Liu and H.-C. Wang, Efficient dominating sets in circulant graphs, {\em Discrete Math.} (2017), \url{http://dx.doi.org/10.1016/j.disc.2017.02}.

\bibitem{E87}
G. Etienne, Perfect codes and regular partitions in graphs and groups, {\em European J. Combin.} {\bf 8} (1987), 139--144.

\bibitem{Etzion11}
T. Etzion, Product constructions for perfect Lee codes, \emph{IEEE Trans. Inform. Theory} {\bf 57} (2011), 7473--7481.

\bibitem{GHT}
A-A. Ghidewon, R. H. Hammack and D. T. Taylor, Total perfect codes in tensor products of graphs, {\em Ars Combin.} {\bf 88} (2008), 129--134. 

\bibitem{Heden}
O.~Heden, A survey of perfect codes, {\em Adv. Math. Commun.} {\bf 2} (2008), 223--247.

\bibitem{HXZ}
H. Huang, B. Xia and S. Zhou, Perfect codes in Cayley graphs, preprint, \url{https://arxiv.org/abs/1609.03755}.


\bibitem{J14}
F. Jarvis, Algebraic Number Theory, Springer, 2014.

\bibitem{KP12}
M. Knor and P. Poto\v{c}nik, Efficient domination in cubic vertex-transitive graphs, {\em European J. Combin.} {\bf 33} (2012), no. 8, 1755--1764.

\bibitem{Kra}
J.~Kratochv\'{i}l, Perfect codes over graphs, 
{\em J. Combin. Theory Ser. B} {\bf 40} (1986), 224--228.


\bibitem{L01}
J. Lee, Independent perfect domination sets in Cayley graphs, {\em J. Graph Theory} {\bf 37} (2001), 213--219.

\bibitem{KM13}
K. Reji Kumar and G. MacGillivray, Efficient domination in circulant graphs, {\em Discrete Math.} {\bf 313} (2013), 767--771. 
 
\bibitem{MBG07}
C. Mart\'{i}nez, R. Beivide and E. Gabidulin, Perfect codes for metrics induced by circulant graphs, {\em IEEE Trans. Inform. Theory} {\bf 53} (2007), 3042--3052.

\bibitem{OPR07}
N. Obradovi\'{c}, J. Peters and G. Ru\v{z}i\'{c}, Efficient domination in circulant graphs with two chord lengths, {\em Inform. Process. Lett.} {\bf 102} (2007), 253--258. 

\bibitem{T04}
S. Terada, Perfect codes in $\SL(2, 2^f)$, {\em European J. Combin.} {\bf 25} (2004), 1077--1085.

\bibitem{vanLint}
J.~H.~van Lint, A survey of perfect codes, {\em Rocky Mountain J. Math.} {\bf 5} (1975), 199--224.
 
\bibitem{Z15}
S. Zhou, Cyclotomic graphs and perfect codes, preprint, \url{http://arxiv.org/abs/1502.03272}.

\bibitem{Z16}
S. Zhou, Total perfect codes in Cayley graphs, {\em Des. Codes Cryptogr.} {\bf 81} (2016), 489--504.

\end{thebibliography}
}
\end{document}